\documentclass[a4paper,11pt]{article}
\usepackage[utf8]{inputenc}

\usepackage{graphicx}
\usepackage{geometry}
\usepackage{amsthm}
\usepackage{amssymb}
\usepackage{amsmath}
\usepackage{hyperref}
\usepackage{xcolor}
\usepackage{enumitem}
\usepackage{listings}
\usepackage{complexity}
\usepackage{microtype}
\usepackage[textsize=tiny]{todonotes}

\usepackage[font=small,labelfont=bf]{caption}
\usepackage[font=small,labelfont=normalfont,labelformat=simple]{subcaption}

\graphicspath{{figs/}}

\newtheorem{theorem}{Theorem}

\newtheorem{lemma}[theorem]{Lemma}
\newtheorem{corollary}[theorem]{Corollary}

\DeclareMathOperator{\conv}{conv}

\def\RR{\mathbb{R}}
\def\NN{\mathbb{N}}

\newenvironment{enumL}[2][]
{\begin{enumerate}[series=tests,label=(#2\arabic*),#1]}
{\end{enumerate}}

\def\inst#1{$^{#1}$}

\date{}

\title{Holes and islands in random point sets\footnote{An extended abstract of this paper appeared in the Proceedings of the 36th International Symposium on Computational Geometry (SoCG 2020).}}

\begin{document}

\author{Martin Balko\inst{1}\thanks{M. Balko was supported by the grant no.~18-19158S of the Czech Science Foundation (GA\v{C}R), by the Center for Foundations of Modern Computer Science (Charles University project UNCE/SCI/004), and by the PRIMUS/17/SCI/3 project of Charles University.
This article is part of a project that has received funding from the European Research Council (ERC) under the European Union's Horizon 2020 research and innovation programme (grant agreement No 810115).} 
\and
Manfred Scheucher\inst{2}\thanks{M.\ Scheucher was supported by DFG Grant FE~340/12-1.}
\and
Pavel Valtr\inst{1,3}\thanks{P. Valtr was supported by the grant no.~18-19158S of the Czech Science Foundation (GA\v{C}R) and by the PRIMUS/17/SCI/3 project of Charles University.}
}

\maketitle

\begin{center}
{\footnotesize
\inst{1} 
Department of Applied Mathematics, \\
Faculty of Mathematics and Physics, Charles University, Prague, Czech Republic \\
\texttt{balko@kam.mff.cuni.cz}
\\\ \\
\inst{2} 
Institut f\"ur Mathematik, Technische Universit\"at Berlin, Germany\\
\texttt{scheucher@math.tu-berlin.de}
\\\ \\
\inst{3} 
Department of Computer Science, ETH Z\"urich, Switzerland
}
\end{center}

\begin{abstract}
For $d\in\mathbb{N}$, let $S$ be a set of points in $\mathbb{R}^d$ in general position.
A set $I$ of $k$ points from $S$ is a \emph{$k$-island} in $S$ if the convex hull $\conv(I)$ of $I$ satisfies $\conv (I) \cap S = I$.
A $k$-island in $S$ in convex position is a \emph{$k$-hole} in~$S$.

For $d,k\in\mathbb{N}$ and a convex body $K\subseteq\mathbb{R}^d$ of volume $1$, let $S$ be a set of $n$ points chosen uniformly and independently at random from~$K$.
We show that the expected number of $k$-holes in $S$ is in $O(n^d)$.
Our estimate improves and generalizes all previous bounds. 
In particular, we estimate the expected number of empty simplices in $S$ by $2^{d-1}\cdot d!\cdot\binom{n}{d}$.
This is tight in the plane up to a lower-order term.

Our method gives an asymptotically tight upper bound $O(n^d)$ even in the much more general setting, where we estimate the expected number of $k$-islands in $S$.
\end{abstract}

\section{Introduction}

For $d \in \mathbb{N}$, let $S$ be a finite set of points in $\mathbb{R}^d$.
The set $S$ is in \emph{general position} if, for every $k=1,\dots,d-1$, no $k+2$ points of $S$ lie in an affine $k$-dimensional subspace of $\mathbb{R}^d$.
A set $H$ of $k$ points from $S$ is a \emph{$k$-hole} in $S$ if $H$ is in convex position and the interior of the convex hull $\conv(H)$ of $H$ does not contain any point from $S$; see Figure~\ref{fig:preliminaries} for an illustration in the plane.
We say that a subset of $S$ is a \emph{hole} in $S$ if it is a $k$-hole in $S$ for some integer $k$.

\begin{figure}[htb]
\centering  
\hbox{}
\hfill
\begin{subfigure}[b]{.3\textwidth}
\centering
\includegraphics[page=1]{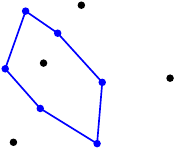}
\caption{}
\label{fig:preliminaries_1}  
\end{subfigure}
\hfill
\begin{subfigure}[b]{.3\textwidth}
\centering
\includegraphics[page=3]{figs/preliminaries}
\caption{}
\label{fig:preliminaries_2}  
\end{subfigure}
\hfill
\begin{subfigure}[b]{.3\textwidth}
\centering
\includegraphics[page=2]{figs/preliminaries}
\caption{}
\label{fig:preliminaries_3}  
\end{subfigure}
\hfill
\hbox{}
  
\caption{
\subref{fig:preliminaries_1}~A $6$-tuple of points in convex position in a planar set $S$ of 10 points.
\subref{fig:preliminaries_2}~A $6$-hole in $S$.
\subref{fig:preliminaries_3}~A $6$-island in $S$ whose points are not in convex position.
}
\label{fig:preliminaries}
\end{figure}

Let $h(k)$ be the smallest positive integer $N$ such that every set of $N$ points in general position in the plane contains a $k$-hole.
In the 1970s, Erd\H{o}s~\cite{Erdos1978} asked whether the number $h(k)$ exists for every $k \in \NN$.
It was shown in the 1970s and 1980s that $h(4)=5$, $h(5)=10$~\cite{Harborth1978}, and that $h(k)$ does not exist for every $k \geq 7$~\cite{Horton1983}.
That is, while every sufficiently large set contains a $4$-hole and a $5$-hole, Horton constructed arbitrarily large sets with no 7-holes.
His construction was generalized to so-called \emph{Horton sets} by Valtr~\cite{VALTR1992b}.
The existence of 6-holes in every sufficiently large point set remained open until 2007, when Gerken~\cite{Gerken2008} and Nicolas~\cite{Nicolas2007} independently showed that $h(6)$ exists; see also~\cite{Valtr2009}.

These problems were also considered in higher dimensions.
For $d \geq 2$, let $h_d(k)$ be the smallest positive integer $N$ such that every set of $N$ points in general position in~$\RR^d$ contains a $k$-hole.
In particular, $h_2(k) = h(k)$ for every $k$.
Valtr~\cite{VALTR1992b} showed that $h_d(k)$ exists for $k \le 2d+1$ but it does not exist for $k > 2^{d-1} (P(d-1)+1)$,  where $P(d-1)$ denotes the product of the first $d-1$ prime numbers.
The latter result was obtained by constructing multidimensional analogues of the Horton sets.

After the existence of $k$-holes was settled, counting the minimum number $H_k(n)$ of $k$-holes in any set of $n$ points in the plane in general position attracted a lot of attention.
It is known, and not difficult to show, that $H_3(n)$ and $H_4(n)$ are in $\Omega(n^2)$. 
The currently best known lower bounds on $H_3(n)$ and $H_4(n)$ were proved in~\cite{5holes_JCTA2020}.
The best known upper bounds are due to B\'{a}r\'{a}ny and Valtr~\cite{BaranyValtr2004}.
Altogether, these estimates are
\[n^2 + \Omega(n\log^{2/3}n) \le  H_3(n) \le 1.6196 n^2 +o(n^2)\]
and
\[\tfrac{n^2}{2} +\Omega(n\log^{3/4}n)  \le  H_4(n) \le 1.9397 n^2 +o(n^2).\]
For $H_5(n)$ and $H_6(n)$, the best quadratic upper bounds can be found in~\cite{BaranyValtr2004}.
The best lower bounds, however, are only $H_5(n) \geq \Omega(n \log^{4/5}n)$~\cite{5holes_JCTA2020}
and $H_6(n) \geq \Omega(n)$~\cite{Valtr2012}. 
For more details, we also refer to the second author's dissertation~\cite{Scheucher2019_dissertation}.

The quadratic upper bound on $H_3(n)$ can be also obtained using random point sets.
For $d \in \mathbb{N}$, a \emph{convex body} in $\RR^d$ is a compact convex set in $\RR^d$ with a nonempty interior. 
Let $k$ be a positive integer and let $K \subseteq \RR^d$ be a convex body with $d$-dimensional Lebesgue measure $\lambda_d(K)=1$.
We use $EH^K_{d,k}(n)$ to denote the expected number of $k$-holes in sets of $n$ points chosen independently and uniformly at random from $K$. 
The quadratic upper bound on $H_3(n)$ then also follows from the following bound of B\'{a}r\'{a}ny and F\"{u}redi~\cite{BaranyFueredi1987} on the expected number of $(d+1)$-holes:
\begin{equation}
\label{boundBaranyFuredi}
EH_{d,d+1}^K(n) \leq (2d)^{2d^2} \cdot \binom{n}{d}
\end{equation}
 for any $d$ and $K$.
In the plane, B\'{a}r\'{a}ny and F\"{u}redi~\cite{BaranyFueredi1987} proved $EH_{2,3}^K(n) \le 2n^2+O(n \log n)$ for every~$K$. This bound was later slightly improved by Valtr~\cite{Valtr1995}, who showed 
\begin{equation}
\label{eq-pavel}
EH^K_{2,3}(n) \le 4\binom{n}{2}
\end{equation}
for any $K$.
In the other direction, every set of $n$ points in $\mathbb{R}^d$ in general position contains at least $\binom{n-1}{d}$ $(d+1)$-holes~\cite{BaranyFueredi1987,katMe88}.

The expected number $EH_{2,4}^K(n)$ of $4$-holes in random sets of $n$ points in the plane was considered by Fabila-Monroy, Huemer, and Mitsche~\cite{mhm15}, who showed
\begin{equation}
\label{eq-MHM}
EH_{2,4}^K(n) \leq 18\pi D^2 n^2 + o(n^2)
\end{equation}
for any $K$, where $D=D(K)$ is the diameter of $K$.
Since we have $D \geq 2/\sqrt{\pi}$, by the Isodiametric inequality~\cite{evansGariepy15}, the leading constant in~\eqref{eq-MHM} is at least $72$ for any $K$. 

In this paper, we study the number of $k$-holes in random point sets in $\RR^d$.
In particular, we obtain results that imply quadratic upper bounds on $H_k(n)$ for any fixed $k$ and that both strengthen and generalize the bounds by B\'{a}r\'{a}ny and F\"{u}redi~\cite{BaranyFueredi1987}, Valtr~\cite{Valtr1995}, and Fabila-Monroy, Huemer, and Mitsche~\cite{mhm15}.

\section{Our results}
\label{sec-ourResults}

Throughout the whole paper we only consider point sets in~$\mathbb{R}^d$ that are finite and in general position.

\subsection{Islands and holes in random point sets}

First, we prove a result that implies the estimate $O(n^d)$ on the minimum number of $k$-holes in a set of $n$ points in $\RR^d$ for any fixed $d$ and $k$.
In fact, we prove the upper bound $O(n^d)$ even for so-called $k$-islands, which are also frequently studied in discrete geometry.
A set $I$ of $k$ points from a point set $S \subseteq \RR^d$ is a \emph{$k$-island} in $S$ if $\conv (I) \cap S = I$; see part~(c) of Figure~\ref{fig:preliminaries}.
Note that $k$-holes in $S$ are exactly those $k$-islands in~$S$ that are in convex position.
A subset of $S$ is an \emph{island} in $S$ if it is a $k$-island in $S$ for some integer $k$.

\begin{theorem}
\label{thm:islands_2d}
Let $d \geq 2$ and $k \geq d+1$ be integers and let $K$ be a convex body in~$\mathbb{R}^d$ with $\lambda_d(K)=1$.
If $S$ is a set of $n \geq k$ points chosen uniformly and independently at random from~$K$, then the expected number of $k$-islands in $S$ is at most
\[2^{d-1}\cdot \left(2d^{2d-1}\binom{k}{\lfloor d/2 \rfloor}\right)^{k-d-1} \cdot (k-d) \cdot \frac{n(n-1) \cdots (n-k+2)}{(n-k+1)^{k-d-1}},\]
which is in $O(n^d)$ for any fixed $d$ and $k$.
\end{theorem}

Theorem~\ref{thm:islands_2d} gives the first nontrivial upper bound on the expected number of $k$-islands in a random point set in $\mathbb{R}^d$.
Moreover, the bound in Theorem~\ref{thm:islands_2d} is tight up to a constant multiplicative factor that depends on $d$ and $k$, as, for any fixed $k \geq d$, every set $S$ of $n$ points in~$\RR^d$ in general position contains at least $\Omega(n^d)$ $k$-islands.
To see this, observe that any $d$-tuple $T$ of points from $S$ determines a $k$-island with $k-d$ closest points to the hyperplane spanned by $T$ (ties can be broken by, for example, taking points with lexicographically smallest coordinates), as $S$ is in general position and thus $T$ is a $d$-hole in $S$.
Any such $k$-tuple of points from $S$ contains $\binom{k}{d}$ $d$-tuples of points from $S$ and thus we have at least $\binom{n}{d}/\binom{k}{d} \in \Omega(n^d)$ $k$-islands in $S$.

Thus, by Theorem~\ref{thm:islands_2d}, random point sets in $\RR^d$ asymptotically achieve the minimum number of $k$-islands.
This is in contrast with the fact that, unlike Horton sets, they contain arbitrarily large holes.
Quite recently, Balogh, Gonz\'alez-Aguilar, and Salazar~\cite{BaloghGAS2013} showed that the expected number of vertices of the largest hole in a set of $n$ random points chosen independently and uniformly over a convex body in the plane is in $\Theta(\log n/(\log \log n))$.

Since every $k$-hole is a $k$-island, Theorem~\ref{thm:islands_2d} implies the upper bound $O(n^d)$ on the expected number of $k$-holes in a random point set in $\mathbb{R}^d$.
This is the first nontrivial estimate of this type whenever $k>d+1$ and  $(d,k) \neq (2,4)$.
As already mentioned, we also obtain the upper bound $O(n^d)$ on the minimum number of $k$-holes in a set of $n$ points in $\mathbb{R}^d$.

By modifying the proof of Theorem~\ref{thm:islands_2d}, we obtain the following variant of Theorem~\ref{thm:islands_2d} for $k$-holes with a better leading constant.

\begin{theorem}
\label{thm:holes_2d}
Let $d \geq 2$ and $k \geq d+1$ be integers and let $K$ be a convex body in~$\mathbb{R}^d$ with $\lambda_d(K)=1$.
If $S$ is a set of $n \geq k$ points chosen uniformly and independently at random from~$K$, then the expected number $EH^K_{d,k}(n)$ of $k$-holes in $S$ is in $O(n^d)$ for any fixed $d$ and $k$. 
More precisely, 
\[EH^K_{d,k}(n) \leq 2^{d-1}\cdot \left(2d^{2d-1}\binom{k}{\lfloor d/2 \rfloor}\right)^{k-d-1} \cdot \frac{n(n-1) \cdots (n-k+2)}{(k-d-1)! \cdot (n-k+1)^{k-d-1}}.\]
\end{theorem}

We now show that the estimate from Theorem~\ref{thm:holes_2d} generalizes and improves the previously known bounds~\eqref{boundBaranyFuredi}, \eqref{eq-pavel}, and~\eqref{eq-MHM}.
For $d=2$ and $k=4$, Theorem~\ref{thm:holes_2d} implies $EH^K_{2,4}(n) \leq 128 \cdot n^2 + o(n^2)$ for any $K$, which is a worse estimate than~\eqref{eq-MHM} if the diameter of $K$ is at most $8/(3\sqrt{\pi}) \simeq 1.5$.
However, the proof of Theorem~\ref{thm:holes_2d} can be modified to give $EH^K_{2,4}(n) \leq 12n^2 + o(n^2)$ for any $K$, which is always better than~\eqref{eq-MHM}; see the final remarks in Section~\ref{sec:islands_in_Rd}. 
We believe that the leading constant in $EH^K_{2,4}(n)$ can be estimated even more precisely and we hope to discuss this direction in future work.

In the case $k=d+1$, the bound in Theorem~\ref{thm:holes_2d} simplifies to the following estimate on the expected number of $(d+1)$-holes (also called \emph{empty simplices}) in random sets of $n$ points in~$\RR^d$.

\begin{corollary}
\label{thm:d_simplices}
Let $d \geq 2$ be an integer and let $K$ be a convex body in $\mathbb{R}^d$ with $\lambda_d(K)=1$.
If $S$ is a set of $n$ points chosen uniformly and independently at random from~$K$, then the expected number of $(d+1)$-holes in $S$ 
satisfies
\[EH^K_{d,d+1}(n) \leq 2^{d-1} \cdot d! \cdot \binom{n}{d}.\]
\end{corollary}

Corollary~\ref{thm:d_simplices} is stronger than 
the bound~\eqref{boundBaranyFuredi} by B\'{a}r\'{a}ny and F\"{u}\-redi~\cite{BaranyFueredi1987} and, 
in the planar case, coincides with the bound~\eqref{eq-pavel} by Valtr~\cite{Valtr1995}.
It follows from a very recent result by Reitzner and Temesvari \cite{ReitznerTemesvari2019} 
that the bound from Corollary~\ref{thm:d_simplices} is tight for $d=2$ up to a lower-order term.
Their results also give tight bounds on $EH^K_{d,k}(n)$ if $d \ge 3$ and $K$ is an ellipsoid.
No tight bounds on $EH^K_{d,d+1}(n)$ are known if $d \geq 3$ and $K$ is not an ellipsoid.

We also consider islands of all possible sizes and show that their expected number is in $2^{\Theta\left(n^{(d-1)/(d+1)}\right)}$.

\begin{theorem}
\label{thm:exponential}
Let $d \geq 2$ be an integer and let $K$ be a convex body in $\RR^d$ with $\lambda_d(K)=1$.
Then there are constants $C_1=C_1(d)$, $C_2=C_2(d)$, and $n_0=n_0(d)$ such that for every set $S$ of $n \geq n_0$ points chosen uniformly and  independently at random from $K$ the expected number $E^K_d$ of islands in $S$ satisfies
\[2^{C_1 \cdot n^{(d-1)/(d+1)}} \leq E^K_d \leq 2^{C_2 \cdot n^{(d-1)/(d+1)}}.\]
\end{theorem}

Since each island in $S$ has at most $n$ points, there is a $k \in \{1,\dots,n\}$ such that the expected number of $k$-islands in $S$ is at least $(1/n)$-fraction of the expected number of all islands, which is still in $2^{\Omega(n^{(d-1)/(d+1)})}$. 
This shows that the expected number of $k$-islands can become asymptotically much larger than $O(n^d)$ if $k$ is not fixed.

\subsection{Islands and holes in \texorpdfstring{$d$}{d}-Horton sets}

To our knowledge, Theorem~\ref{thm:islands_2d} is the first nontrivial upper bound on the minimum number of $k$-islands a point set in $\RR^d$ with $d>2$ can have.
For $d=2$, Fabila-Monroy and Huemer~\cite{FabilaMonroyHuemer2012} showed that, for every fixed $k \in\NN$, the Horton sets with $n$ points contain only $O(n^2)$ $k$-islands.
For $d>2$, Valtr~\cite{VALTR1992b} introduced a $d$-dimensional analogue of Horton sets.
The exact definition is rather technical; see Section~\ref{sec:prop-Horton}.
Perhaps surprisingly, these sets contain asymptotically more than $O(n^d)$ $k$-islands for $k \geq d+1$.
For each $k$ with $d+1 \leq k \leq 3 \cdot 2^{d-1}$, they even contain asymptotically more than $O(n^d)$ $k$-holes.

\begin{theorem}
\label{prop-Horton}
Let $d \geq 2$ and $k$ be fixed positive integers.
Then every $d$-dimensional Horton set $H$ with $n$ points contains at least $\Omega(n^{\min\{2^{d-1},k\}})$ $k$-islands in $H$.
If $k \leq 3 \cdot 2^{d-1}$, then $H$ even contains at least $\Omega(n^{\min\{2^{d-1},k\}})$ $k$-holes in $H$.
\end{theorem}

\section{Proofs of Theorem~\ref{thm:islands_2d} and Theorem~\ref{thm:holes_2d}}
\label{sec:islands_in_Rd}

Let $d$ and $k$ be positive integers and let $K$ be a convex body in~$\mathbb{R}^d$ with $\lambda_d(K)=1$.
Let $S$ be a set of $n$ points chosen uniformly and independently at random from~$K$.
Note that $S$ is in general position with probability $1$.
We assume $k \geq d+1$, as otherwise the number of $k$-islands in $S$ is trivially $\binom{n}{k}$ in every set of $n$ points in $\RR^d$ in general position.
We also assume $d \geq 2$ and $n \geq k$, as otherwise the number of $k$-islands is trivially $n-k+1$ and $0$, respectively, in every set of $n$ points in $\mathbb{R}^d$.

First, we prove Theorem~\ref{thm:islands_2d} by showing that the expected number of $k$-islands in $S$ is at most 
\[2^{d-1}\cdot \left(2d^{2d-1}\binom{k}{\lfloor d/2 \rfloor}\right)^{k-d-1} \cdot (k-d) \cdot \frac{n(n-1) \cdots (n-k+2)}{(n-k+1)^{k-d-1}},\]
which is in $O(n^d)$ for any fixed $d$ and $k$.
At the end of this section, we improve the bound for $k$-holes, which will prove Theorem~\ref{thm:holes_2d}.

Let $Q$ be a set of $k$ points from $S$.
We first introduce a suitable unique ordering $q_1,\dots,q_k$ of points from~ $Q$.
First, we take a set $D$ of $d+1$ points from $Q$ that determine a simplex $\Delta$ with largest volume among all $(d+1)$-tuples of points from $Q$.
Let $q_1q_2$ be the longest edge of $\Delta$ with $q_1$ lexicographically smaller than $q_2$ and let $a$ be the number of points from $Q$ inside $\Delta$.
For every $i=2,\dots,d$, let $q_{i+1}$ be the furthest point from $D\setminus\{q_1,\dots,q_i\}$ to ${\rm aff}(q_1,\dots,q_i)$.
Next, we let $q_{d+2},\dots,q_{d+a+1}$ be the $a$ points of $Q$ inside $\Delta$ ordered lexicographically.
The remaining $k-d-a-1$ points $q_{d+a+2},\dots,q_k$ from $Q$ lie outside of $\Delta$ and we order them so that, for every $i=1,\dots,k-a-d-1$, the point $q_{d+a+i+1}$ is closest to ${\rm conv}(\{q_1,\dots,\allowbreak q_{d+a+i}\})$ among the points $q_{d+a+i+1},\dots,q_k$.
In case of a tie in any of the conditions, we choose the point with lexicographically smallest coordinates.
Note, however, that a tie occurs with probability $0$.

Clearly, there is a unique such ordering $q_1,\dots,q_k$ of $Q$.
We call this ordering the \emph{canonical $(k,a)$-ordering} of $Q$.
To reformulate, an ordering $q_1,\dots,q_k$ of $Q$ is the canonical $(k,a)$-ordering of~$Q$ if and only if the following five conditions are satisfied:
\begin{enumL}{L}
\item\label{item-canonical0} The $d$-dimensional simplex $\Delta$,  with vertices $q_1,\dots,q_{d+1}$ has the largest $d$-dimen\-sional Lebesgue measure among all $d$-dimensional simplices spanned by points from~$Q$.

\item\label{item-canonical1} 
For every $i=1,\dots,d-1$, the point $q_{i+1}$ has the largest distance among all points from $\{q_{i+1},\dots,q_d\}$ to the $(i-1)$-dimensional affine subspace ${\rm aff}(q_1,\dots,q_i)$ spanned by $q_1,\dots,q_i$.
Moreover, $q_1$ is lexicographically smaller than $q_2$.

\item\label{item-canonical4} For every $i=1,\dots,d-1$, the distance between $q_{i+1}$ and ${\rm aff}(q_1,\dots,q_{i})$ is at least as large as the distance between $q_{d+1}$ and ${\rm aff}(q_1,\dots,q_i)$.
Also, the distance between $q_1$ and $q_2$ is at least as large as the distance between $q_{d+1}$ and any $q_i$ with $i \in \{1,\dots,d\}$.

\item\label{item-canonical2} The points $q_{d+2},\dots,q_{d+a+1}$ lie inside $\Delta$ and are ordered lexicographically.

\item\label{item-canonical3} The points $q_{d+a+2},\dots,q_k$ lie outside of $\Delta$.
For every $i=1,\dots,k-a-d-1$, the point $q_{d+a+i+1}$ is closest to ${\rm conv}(\{q_1,\dots,\allowbreak q_{d+a+i}\})$ among the points $q_{d+a+i+1},\dots,q_k$.
\end{enumL}

Figure~\ref{fig:canlab} gives an illustration in~$\RR^2$.
We note that the conditions~\ref{item-canonical1} and~\ref{item-canonical4} can be merged together.
However, later in the proof, we use the fact that the probability that the points from $Q$ satisfy the condition~\ref{item-canonical1} equals $1/d!$, so we stated the two conditions separately.

\begin{figure}[htb]
  \centering
    \includegraphics{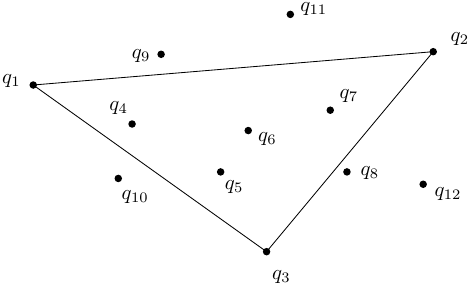}
  \caption{An illustration of the canonical $(k,a)$-ordering of a planar point set~$Q$. 
  Here we have $k=12$ points and $a=4$ of the points 
  lie inside the largest area triangle $\Delta$ with vertices $q_1,q_2,q_3$.}
  \label{fig:canlab}
\end{figure}

Before going into details, we first give a high-level overview of the proof of Theorem~\ref{thm:islands_2d}. 
First, we prove an $O(1/n^{a+1})$ bound on the probability that
$\Delta$ contains precisely the points $p_{d+2},\ldots,p_{d+1+a}$ from~$S$ and satisfies~\ref{item-canonical3} (Lemma~\ref{lem-probabilityInside}),
which means that the points $p_1,\ldots,p_{d+1+a}$ determine an island in $S$.
Next, for $i=d+2+a,\ldots,k$, 
we show that, 
conditioned on the fact that 
the $(i-1)$-tuple $(p_1,\ldots,p_{i-1})$ determines an island in $S$ 
in the canonical $(k,a)$-ordering,
the $i$-tuple $(p_1,\ldots,p_{i})$ determines an island in $S$ 
in the canonical $(k,a)$-ordering
with probability $O(1/n)$ (Lemma~\ref{lem-probabilityOutside}).
Then it immediately follows that the probability
that $I$ determines a $k$-island in $S$ with the desired properties is at most $
O \left( 1/n^{a+1} \cdot (1/n)^{k-(d+1+a)} \right) = O(n^{d-k})$.
Since there are $n \cdot (n-1)\cdots(n-k+1)=O(n^k)$ 
possibilities to select such an ordered subset $I$
and each $k$-island in $S$ is counted at most $k!$ times,
we obtain the desired bound $O \left( n^k \cdot n^{d-k} \cdot k!  \right) = O(n^d)$
on the expected number of $k$-islands in~$S$.

We now proceed with the proof.
Let $p_1,\dots,p_k$ be points from $S$ in the order in which they are drawn from $K$.
We use $\Delta$ to denote the $d$-dimensional simplex with vertices $p_1,\dots,p_{d+1}$.
We eventually show that the probability that $p_1,\dots,p_k$ is the canonical $(k,a)$-ordering of a $k$-island in $S$ for some $a$ is at most $O(1/n^{k-d})$.
First, however, we need to state some notation and prove some auxiliary results.

Consider the points $p_1,\dots,p_d$.
Without loss of generality, we can assume that, for each $i=1,\dots,d$, the point $p_i$ has the last $d-i+1$ coordinates equal to zero.
Otherwise we apply a suitable isometry to~$S$.
Then, for every $i=1,\dots,d$, the distance between $p_{i+1}$ and the $(i-1)$-dimensional affine subspace spanned by $p_1,\dots,p_i$ is equal to the absolute value of the $i$th coordinate of $p_{i+1}$.
Moreover, after applying a suitable rotation, we can also assume that the first coordinate of each of the points $p_1,\dots,p_d$ is nonnegative.

Let $\Delta_0$ be the $(d-1)$-dimensional simplex with vertices $p_1,\dots,p_d$ and let $H$ be the hyperplane containing $\Delta_0$.
Note that, according to our assumptions about $p_1,\dots,p_d$, we have $H = \{(x_1,\dots,x_d) \in \RR^d \colon x_d=0\}$.
Let $B$ be the set of points $(x_1,\dots,x_d) \in \RR^d$ that satisfy the following three conditions:
\begin{enumerate}[label=(\roman*)]
\item $x_1 \geq 0$,
\item $|x_i|$ is at most as large as the absolute value of the $i$th coordinate of $p_{i+1}$ for every $i \in \{1,\dots,d-1\}$, and 
\item $|x_d| \leq d/\lambda_{d-1}(\Delta_0)$.
\end{enumerate}
See Figures~\ref{fig:simplices_2d} and~\ref{fig:simplices_3d} for illustrations in $\RR^2$ and $\RR^3$, respectively. 
Observe that $B$ is a $d$-dimensional axis-parallel box.
For $h \in \RR$, we use $I_h$ to denote the intersection of $B$ with the hyperplane $x_d=h$.

\begin{figure}[htb]
\centering  
\hbox{}
\hfill
\begin{subfigure}[b]{.35\textwidth}
\centering
    \includegraphics{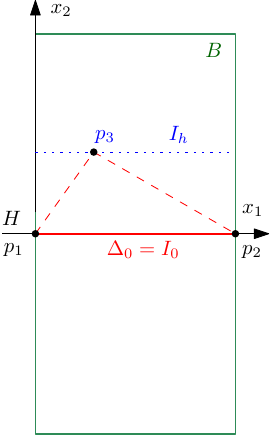}
\caption{}
\label{fig:simplices_2d}  
\end{subfigure}
\hfill
\begin{subfigure}[b]{.5\textwidth}
\centering
    \includegraphics{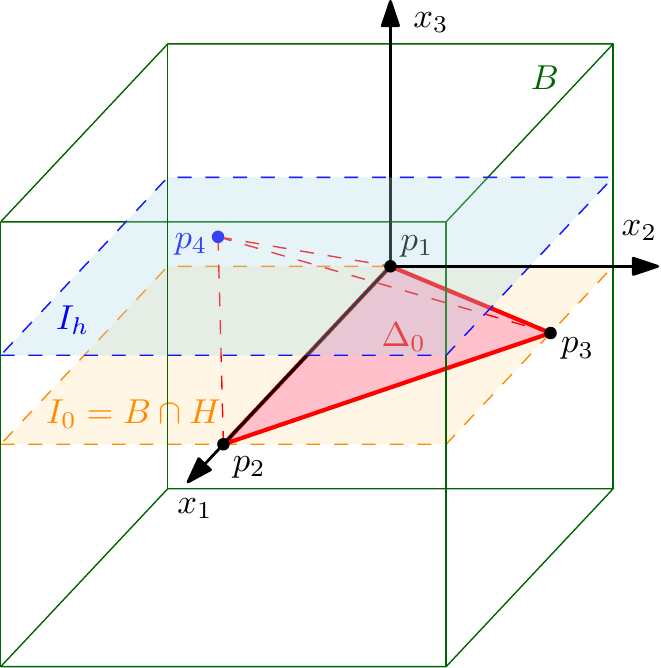}
\caption{}
\label{fig:simplices_3d}  
\end{subfigure}
\hfill
\hbox{}
  
\caption{
An illustration of the proof of Theorem~\ref{thm:islands_2d} in~\subref{fig:simplices_2d}~$\RR^2$ and \subref{fig:simplices_3d}~$\RR^3$.
}
\end{figure}

Having fixed $p_1,\dots,p_d$, we now try to restrict possible locations of the points $p_{d+1},\dots,p_k$, one by one, so that $p_1,\dots,p_k$ is the canonical $(k,a)$-ordering of a $k$-island in $S$ for some $a$.
First, we observe that the position of the point $p_{d+1}$ is restricted to $B$.

\begin{lemma}
\label{lem-containment}
If $p_1,\dots,p_{d+1}$ satisfy condition~\ref{item-canonical4}, then $p_{d+1}$ lies in the box~$B$.
\end{lemma}
\begin{proof}
Let $p_{d+1}=(x_1,\dots,x_d)$.
According to our choice of points $p_1,\dots,p_d$ and from the assumption that $p_1,\dots,p_d$ satisfy~\ref{item-canonical4}, we get $x_1 \geq 0$ and also that $|x_i|$ is at most as large as the absolute value of the $i$th coordinate of $p_{i+1}$ for every $i \in \{1,\dots,d-1\}$.

It remains to show that $|x_d| \leq d/\lambda_{d-1}(\Delta_0)$.
The simplex $\Delta$ spanned by $p_1,\dots,p_{d+1}$ is contained in the convex body $K$, as $p_1,\dots,p_{d+1} \in K$ and $K$ is convex.
Thus $\lambda_d(\Delta) \leq \lambda_d(K) = 1$. 
On the other hand, the volume $\lambda_d(\Delta)$ equals $\lambda_{d-1}(\Delta_0) \cdot h/d$, where $h$ is the distance between $p_{d+1}$ and the hyperplane $H$ containing $\Delta_0$.
According to our assumptions about $p_1,\dots,p_d$, the distance $h$ equals $|x_d|$.
Since $\lambda_d(\Delta) \leq 1$, it follows that $|x_d| = h \leq d/\lambda_{d-1}(\Delta_0)$ and thus $p_{d+1} \in B$.
\end{proof}

The following auxiliary lemma gives an identity that is needed later.
We omit the proof, which can be found, for example, in~\cite[Section~1]{aar99}.

\begin{lemma}[\cite{aar99}]
\label{lem-integral}
For all nonnegative integers $a$ and $b$, we have
\[\int_0^1 x^a(1-x)^b \;{\rm d}x = \frac{a! \; b!}{(a+b+1)!}\,.\]
\end{lemma}

We will also use the following result, called the \emph{Asymptotic Upper Bound Theorem}~\cite{mat02}, that estimates the maximum number of facets in a polytope.

\begin{theorem}[Asymptotic Upper Bound Theorem~\cite{mat02}]
\label{thm-upperBoundThm}
For every integer $d \geq 2$, a $d$-dimensional convex polytope with $N$ vertices has at most $2\binom{N}{\lfloor d/2\rfloor}$ facets.
\end{theorem}

Let $a$ be an integer satisfying $0 \leq a \leq k-d-1$ and let $E_a$ be the event that $p_1,\dots,p_k$ is the canonical $(k,a)$-ordering such that $\{p_1,\dots,p_{d+a+1}\}$ is an island in $S$.
To estimate the probability that $p_1,\dots,p_k$ is the canonical $(k,a)$-ordering of a $k$-island in $S$, we first find an upper bound on the conditional probability of~$E_a$, conditioned on the event $L_2$ that $p_1,\dots,p_d$ satisfy~\ref{item-canonical1}.

\begin{lemma}
\label{lem-probabilityInside}
For every $a \in \{0,\dots,k-d-1\}$, the probability $\Pr[E_a \mid L_2]$ is at most
\[\frac{2^{d-1}\cdot d!}{(k-a-d-1)!\cdot(n-k+1)^{a+1}} .\]
\end{lemma}
\begin{proof}

It follows from Lemma~\ref{lem-containment} that, in order to satisfy~\ref{item-canonical4}, the point $p_{d+1}$ must lie in the box~$B$.
In particular, $p_{d+1}$ is contained in $I_h \cap K$ for some real number $h \in [-d/\lambda_{d-1}(\Delta_0),d/\lambda_{d-1}(\Delta_0)]$.
If $p_{d+1} \in I_h$, then the simplex $\Delta = \conv(\{p_1,\dots,p_{d+1}\})$ has volume $\lambda_d(\Delta) = \lambda_{d-1}(\Delta_0)\cdot|h|/d$ and the $a$ points $p_{d+2},\dots,p_{d+a+1}$ satisfy~\ref{item-canonical2} with probability 
\[\frac{1}{a!}\cdot \left(\lambda_d(\Delta)\right)^a = \frac{1}{a!}\cdot \left(\frac{\lambda_{d-1}(\Delta_0) \cdot |h|}{d}\right)^a,\] as they all lie in $\Delta \subseteq K$ in the unique order.

In order to satisfy the condition~\ref{item-canonical3},
the $k-a-d-1$ points $p_{d+a+i+1}$, for $i \in \{1,\dots,k-a-d-1\}$, must have increasing distance to ${\rm conv}(\{p_1,\dots,\allowbreak p_{d+a+i}\})$ as the index $i$ increases, which happens with probability $\frac{1}{(k-a-d-1)!}$.
Here we are assuming without loss of generality that all the distances are distinct, as it happens with probability 1.
Since $\{p_1,\dots,p_{d+a+1}\}$ must be an island in $S$, the $n-d-a-1$ points from $S \setminus \{p_1,\dots,p_{d+a+1}\}$ must lie outside of $\Delta$.
If $p_{d+1} \in I_h$, then this happens with probability
\[(\lambda_d(K \setminus \Delta))^{n-d-a-1} = (\lambda_d(K)-\lambda_d(\Delta))^{n-d-a-1} = \left(1 - \frac{\lambda_{d-1}(\Delta_0) \cdot |h|}{d}\right)^{n-d-a-1},\]
as they all lie in $K \setminus \Delta$ and we have $\Delta \subseteq K$ and $\lambda_d(K)=1$.

Altogether, we get that $\Pr[E_a \mid L_2]$ is at most
\[\int\displaylimits_{-d/\lambda_{d-1}(\Delta_0)}^{d/\lambda_{d-1}(\Delta_0)} \frac{\lambda_{d-1}(I_h \cap K)}{a!\cdot (k-a-d-1)!} \cdot \left(\frac{\lambda_{d-1}(\Delta_0) \cdot |h|}{d}\right)^a \cdot \left(1 - \frac{\lambda_{d-1}(\Delta_0) \cdot |h|}{d}\right)^{n-d-a-1} {\rm d} h.
 \]
Since we have $\lambda_{d-1}(I_0) = \lambda_{d-1}(I_h)$ for every $h \in [-d/\lambda_{d-1}(\Delta_0),d/\lambda_{d-1}(\Delta_0)]$, we obtain $\lambda_{d-1}(I_h \cap K) \leq \lambda_{d-1}(I_0)$ and thus $\Pr[E_a \mid L_2]$ is at most
\[\frac{2 \cdot \lambda_{d-1}(I_0)}{a!\cdot (k-a-d-1)!   } \cdot \int\displaylimits_0^{d/\lambda_{d-1}(\Delta_0)} \left(\frac{\lambda_{d-1}(\Delta_0) \cdot h}{d}\right)^a \cdot \left(1 - \frac{\lambda_{d-1}(\Delta_0) \cdot h}{d}\right)^{n-d-a-1} {\rm d} h.\]

By substituting $t=\frac{\lambda_{d-1}(\Delta_0) \cdot h}{d}$, we obtain
\[\Pr[E_a \mid L_2] \leq \frac{2d\cdot \lambda_{d-1}(I_0)}{a!\cdot (k-a-d-1)!\cdot\lambda_{d-1}(\Delta_0)} \cdot  \int_0^1 t^a (1-t)^{n-d-a-1} {\rm d}t.\]
By Lemma~\ref{lem-integral}, the right side in the above inequality equals
\begin{align*}
\frac{2d\cdot \lambda_{d-1}(I_0)}{a!\cdot (k-a-d-1)! \cdot \lambda_{d-1}(\Delta_0)} &\cdot \frac{a! \cdot (n-d-a-1)!}{(n-d)!} \\
&= \frac{2d\cdot\lambda_{d-1}(I_0)}{(k-a-d-1)! \cdot\lambda_{d-1}(\Delta_0)} \cdot \frac{(n-d-a-1)!}{(n-d)!}.
\end{align*}

For every $i=1,\dots,d-1$, let $h_i$ be the distance between the point $p_{i+1}$ and the $(i-1)$-dimensional affine subspace spanned by $p_1,\dots,p_i$.
Since the volume of the box $I_0$ satisfies 
\[\lambda_{d-1}(I_0) = h_1(2h_2)\cdots(2h_{d-1}) = 2^{d-2} \cdot h_1\cdots h_{d-1}\]
and the volume of the $(d-1)$-dimensional simplex $\Delta_0$ is 
\[\lambda_{d-1}(\Delta_0) = \frac{h_1}{1} \cdot \frac{h_2}{2} \cdot \, \cdots \, \cdot \frac{h_{d-1}}{d-1} = \frac{h_1 \cdots h_{d-1}}{(d-1)!},\] 
we obtain $\lambda_{d-1}(I_0) / \lambda_{d-1}(\Delta_0) = 2^{d-2}\cdot (d-1)!$.
Thus
\begin{align*}
\Pr[E_a \mid L_2] &\leq \frac{  2^{d-1} \cdot d!  }{(k-a-d-1)!}    \cdot \frac{(n-d-a-1)!}{(n-d)!} \\
&= \frac{2^{d-1}\cdot d!}{(k-a-d-1)! \cdot (n-d)\cdots(n-d-a)}\\
&\leq \frac{2^{d-1}\cdot d!}{(k-a-d-1)! \cdot    (n-k+1)^{a+1} },
\end{align*}
where the last inequality follows from $a \leq k-d-1$.
\end{proof}

For every $i \in \{d+a+1,\dots,k\}$, let $E_{a,i}$ be the event that $p_1,\dots,p_k$ is the canonical $(k,a)$-ordering such that  $\{p_1,\dots,p_i\}$ is an island in $S$.
Note that in the event $E_{a,i}$ the condition~\ref{item-canonical3} implies that $\{p_1,\dots,p_j\}$ is an island in $S$ for every $j \in \{d+a+1,\dots,i\}$.
Thus we have
\[L_2 \supseteq E_a = E_{a,d+a+1} \supseteq E_{a,d+a+2} \supseteq \cdots \supseteq E_{a,k}.\]
Moreover, the event $E_{a,k}$ says that $p_1,\dots,p_k$ is the canonical $(k,a)$-ordering of a $k$-island in~$S$. 

For $i \in \{d+a+2,\dots,k\}$, we now estimate the conditional probability of $E_{a,i}$, conditioned on $E_{a,i-1}$.

\begin{lemma}
\label{lem-probabilityOutside}
For every $i \in \{d+a+2,\dots,k\}$, we have
\[\Pr[E_{a,i} \mid E_{a,i-1}] \leq \frac{2d^{2d-1} \cdot  \binom{k}{\lfloor d/2 \rfloor}}{n-k+1}.\]
\end{lemma}
\begin{proof}
Let $i \in \{d+a+2,\dots,k\}$ and assume that the event $E_{a,i-1}$ holds.
That is, $p_1,\dots,p_k$ is the canonical $(k,a)$-ordering such that $\{p_1,\dots,p_{i-1}\}$ is an $(i-1)$-island in~$S$.

First, we assume that $\Delta$ is a regular simplex with height $\eta>0$.
At the end of the proof we show that the case when $\Delta$ is an arbitrary simplex follows by applying a suitable affine transformation.

For every $j \in \{1,\dots,d+1\}$, let $F_j$ be the facet $\conv(\{p_1,\dots,\allowbreak p_{d+1}\}\setminus\{p_j\})$ of $\Delta$ and let $H_j$ be the hyperplane parallel to $F_j$ that contains $p_j$.
We use $H^+_j$ to denote the halfspace determined by $H_j$ such that $\Delta \subseteq H_j^+$.
We set $\Delta^* = \cap_{j=1}^{d+1} H^+_j$; see Figures~\ref{fig:Fischer2d} and \ref{fig:Fischer3d} for illustrations in $\RR^2$ and~$\RR^3$, respectively. 
Note that $\Delta^*$ is a $d$-dimensional simplex containing~$\Delta$.
Also, notice that if $x \notin \Delta^*$, then $x \notin H^+_j$ for some $j$ and the distance between $x$ and the hyperplane containing $F_j$ is larger than $\eta$.

\begin{figure}[htb]
\centering  
\hbox{}
\hfill
\begin{subfigure}[b]{.45\textwidth}
\centering
    \includegraphics{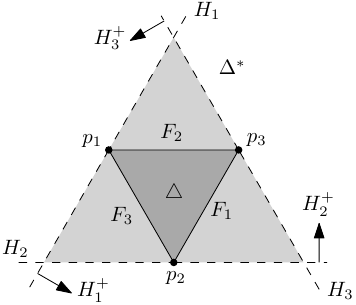}
\caption{}
\label{fig:Fischer2d}  
\end{subfigure}
\hfill
\begin{subfigure}[b]{.5\textwidth}
\centering
    \includegraphics[width=0.95\textwidth]{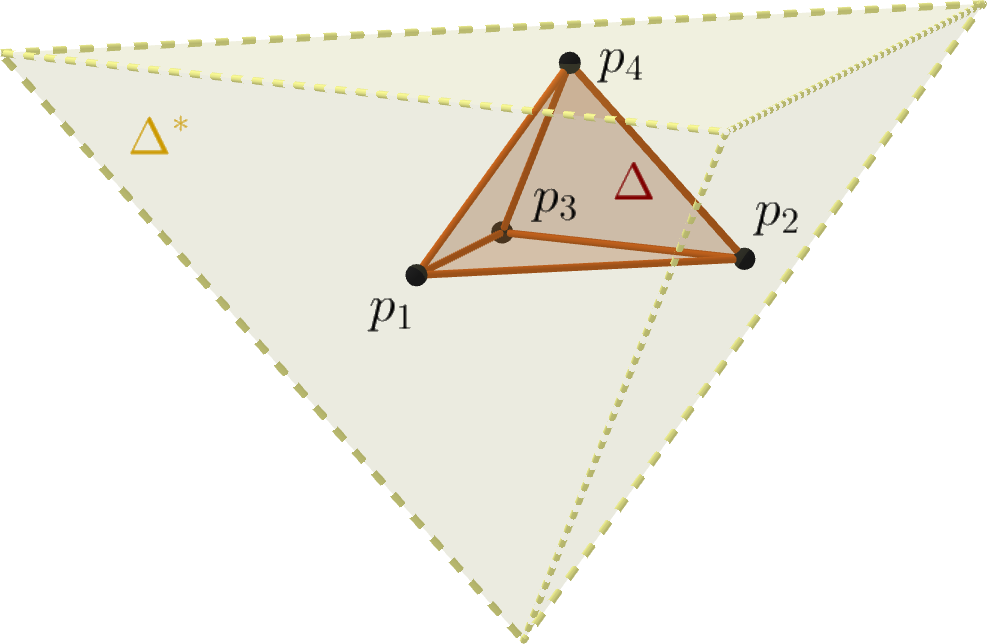} 
\caption{}
\label{fig:Fischer3d}  
\end{subfigure}
\hfill
\hbox{}
  
\caption{
An illustration of 
\subref{fig:Fischer2d}~the simplex $\Delta^*$ in $\RR^2$ and 
\subref{fig:Fischer3d}~in $\RR^3$.
}
\end{figure}

We show that the fact that $p_1,\dots,p_k$ is the canonical $(k,a)$-ordering implies that every point from $\{p_1,\dots,p_k\}$ is contained in $\Delta^*$.
Suppose for contradiction that some point $p \in \{p_1,\dots,p_k\}$ does not lie inside $\Delta^*$.
Then there is a facet $F_j$ of $\Delta$ such that the distance $\eta'$ between $p$ and the hyperplane containing $F_j$ is larger than $\eta$.
Then, however, the simplex $\Delta'$ spanned by vertices of $F_j$ and by $p$ has volume larger than $\Delta$, because
\[\lambda_d(\Delta') = \frac{1}{d} \cdot \lambda_{d-1}(F_j) \cdot \eta' > \frac{1}{d} \cdot \lambda_{d-1}(F_j) \cdot \eta = \lambda_d(\Delta).\]
This contradicts the fact that $p_1,\dots,p_k$ is the canonical $(k,a)$-ordering, as, according to~\ref{item-canonical0}, $\Delta$ has the largest $d$-dimensional Lebesgue measure among all $d$-dimensional simplices spanned by points from $\{p_1,\dots,p_k\}$.

Let $\sigma$ be the barycenter of $\Delta$.
For every point $p \in \Delta^* \setminus \Delta$, the line segment $\sigma p$ intersects at least one facet of $\Delta$.
For every $j \in \{1,\dots,d+1\}$, we use $R_j$ to denote the set of points $p \in \Delta^* \setminus \Delta$ for which the line segment $\sigma p$ intersects the facet $F_j$ of $\Delta$.
Observe that each set $R_j$ is convex and the sets $R_1,\dots,R_{d+1}$ partition $\Delta^* \setminus \Delta$ (up to their intersection of $d$-dimensional Lebesgue measure $0$); see Figure~\ref{fig:Fischer2d_2} for an illustration in the plane.

\begin{figure}[htb]
  \centering
    \includegraphics[page=3]{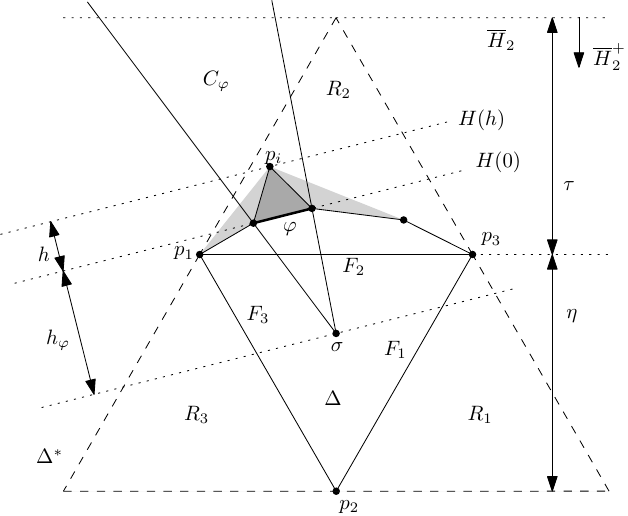}
  \caption{An illustration of the proof of Lemma~\ref{lem-probabilityOutside}.
  In order for $\{p_1,\dots,p_i\}$ to be an $i$-island in~$S$, the light gray part cannot contain points from $S$.
  We estimate the probability of this event from above by the probability that the dark gray simplex $\conv(\varphi \cup \{p_i\})$ contains no point of $S$. 
  Note that the parameters $\eta$ and $\tau$ coincide for $d=2$, as then $\tau = \frac{d^2-1}{d+1}\eta = \eta$.}
  \label{fig:Fischer2d_2}
\end{figure}

Consider the point $p_i$.
Since $p_1,\dots,p_k$ is the canonical $(k,a)$-ordering, the condition~\ref{item-canonical3} implies that $p_i$ lies outside of the polytope $\conv(\{p_1,\dots,p_{i-1}\})$.
To bound the probability $\Pr[E_{a,i} \mid E_{a,i-1}]$, we need to estimate the probability that $\conv(\{p_1,\dots,p_i\}) \setminus \conv(\{p_1,\dots,p_{i-1}\})$ does not contain any point from $S\setminus \{p_1,\dots,p_i\}$, conditioned on $E_{a,i-1}$.
We know that $p_i$ lies in $\Delta^*\setminus \Delta$ and that $p_i \in R_j$ for some $j \in \{1,\dots,d+1\}$.

Since $p_i \notin \conv(\{p_1,\dots,p_{i-1}\})$, there is a facet $\varphi$ of the polytope $\conv(\{p_1,\dots,p_{i-1}\})$ contained in the closure of $R_j$ such that $\sigma p_i$ intersects $\varphi$.
Since $S$ is in general position with probability $1$, we can assume that $\varphi$ is a $(d-1)$-dimensional simplex.
The point $p_i$ is contained in the convex set $C_\varphi$ that contains all points $c \in \mathbb{R}^d$ such that the line segment $\sigma c$ intersects $\varphi$.
We use $H(0)$ to denote the hyperplane containing $\varphi$.
For a positive $r \in \RR$, let $H(r)$ be the hyperplane parallel to $H(0)$ at distance $r$ from $H(0)$ such that $H(r)$ is contained in the halfspace determined by $H(0)$ that does not contain $\conv(\{p_1,\dots,p_{i-1}\})$.
Then we have $p_i \in H(h)$ for some positive $h \in \RR$.

Since $p_i \in K$ and $\varphi \subseteq K$, the convexity of $K$ implies that the simplex $\conv(\varphi \cup \{p_i\})$ has volume $\lambda_d(\conv(\varphi \cup \{p_i\})) \leq \lambda_d(K) = 1$.
Since $\lambda_d(\conv(\varphi \cup \{p_i\})) = \lambda_{d-1}(\varphi) \cdot h/d$, we obtain $h \leq  d/\lambda_{d-1}(\varphi)$.

The point $p_i$ lies in the $(d-1)$-dimensional simplex $C_\varphi \cap H(h)$, which is a scaled copy of~$\varphi$.
We show that 
\begin{equation}
\label{eq-simplex}
\lambda_{d-1}(C_\varphi \cap H(h)) \leq d^{2d-2}\cdot\lambda_{d-1}(\varphi).
\end{equation}
Let $h_\varphi$ be the distance between $H(0)$ and $\sigma$ and, for every $j \in \{1,\dots,d+1\}$, let $\overline{H}_j$ be the hyperplane parallel to $F_j$ containing the vertex $v_j = H_1 \cap \cdots \cap H_{j-1} \cap H_{j+1} \cap \cdots \cap H_{d+1}$ of~$\Delta^*$.
We denote by $\overline{H}^+_j$ the halfspace determined by $\overline{H_j}$ containing $\Delta^*$.
Since $\Delta$ lies on the same side of $H(0)$ as $\sigma$, we see that $h_{\varphi}$ is at least as large as the distance between $\sigma$ and $F_j$, which is $\eta/(d+1)$.
This is because $\tau$ is the length of the line segment $s_j$ connecting $v_j$ with the barycenter of $F_j$ and $h$ is at most as large as the portion of $s_j$ contained between $H(0)$ and $H(h)$. 
Since $p_i$ lies in $\Delta^* \subseteq \overline{H}^+_j$, we see that $h$ is at most as large as the distance $\tau$ between $\overline{H}_j$ and the hyperplane containing the facet $F_j$ of $\Delta$.
Note that $\tau + \eta/(d+1)$ is the distance of the barycenter of $\Delta^*$ and a vertex of $\Delta^*$ and $d\eta/(d+1)$ is the distance of the barycenter of $\Delta^*$ and a facet of $\Delta^*$.
Thus we get $\tau = \frac{d^2\eta}{d+1} - \frac{\eta}{d+1} = \frac{d^2-1}{d+1}\eta$ from the fact that the distance between the barycenter of a $d$-dimensional simplex and any of its vertices is $d$-times as large as the distance between the barycenter and a facet.
Consequently, $h \leq \frac{d^2-1}{d+1}\eta$ and $\frac{\eta}{d+1} \leq h_\varphi$, which implies $h \leq (d^2-1)h_\varphi$.
Thus $C_\varphi \cap H(h)$ is a scaled copy of~$\varphi$ by a factor of size at most~$d^2$.
This gives $\lambda_{d-1}(C_\varphi \cap H(h)) \leq d^{2d-2}\cdot\lambda_{d-1}(\varphi)$.

Since the simplex $\conv(\varphi \cup \{p_i\})$ is a subset of the closure of $\conv(\{p_1,\dots,p_i\}) \setminus \conv(\{p_1,\dots,p_{i-1}\})$, the probability $\Pr[E_{a,i} \mid E_{a,i-1}]$ can be bounded from above by the conditional probability of the event $A_{i,\varphi}$ that $p_i \in C_\varphi \cap K$ and that no point from $S\setminus \{p_1,\dots,p_i\}$ lies in $\conv(\varphi \cup \{p_i\})$, conditioned on $E_{a,i-1}$.
All points from $S\setminus \{p_1,\dots,p_i\}$ lie outside of $\conv(\varphi \cup \{p_i\})$ with probability
\[\left(   1- \frac{\lambda_d(\conv(\varphi \cup \{p_i\}))}{\lambda_d(K \setminus \conv(\{p_1,\ldots,p_{i-1}\}))}    \right)^{n-k},\]
as condition~\ref{item-canonical3} implies that points $p_{i+1},\dots,p_k$ lie outside of $\conv(\varphi \cup \{p_i\})$.
Since $\lambda_d(K \setminus \conv(\{p_1,\ldots,p_{i-1}\}))\leq \lambda_d(K) = 1$, this is bounded from above by
\[(1-\lambda_d(\conv(\varphi \cup \{p_i\})))^{n-k} = \left(1-\frac{\lambda_{d-1}(\varphi) \cdot h}{d}\right)^{n-k}.\]
Since the sets $C_\varphi$ partition $K \setminus \conv(\{p_1,\dots,p_{i-1}\})$ (up to intersections of $d$-dimensional Lebesgue measure $0$) and since $h \leq  d/\lambda_{d-1}(\varphi)$, we have, by the law of total probability,
\begin{align*}
\Pr[E_{a,i} \mid E_{a,i-1}] &\leq \sum_{\varphi}\Pr[A_{i,\varphi} \mid E_{a,i-1}]\\
&\leq \sum_\varphi \int\displaylimits_0^{d/\lambda_{d-1}(\varphi)} \lambda_{d-1}(C_\varphi \cap H(h)) \cdot \left(1-\frac{\lambda_{d-1}(\varphi) \cdot h}{d}\right)^{n-k} \; {\rm d}h.
\end{align*}
The sums in the above expression are taken over all facets $\varphi$ of the convex polytope $\conv(\{p_1,\dots,\allowbreak p_{i-1}\})$.
Using~\eqref{eq-simplex}, we can estimate $\Pr[E_{a,i} \mid E_{a,i-1}]$ from above by
\[d^{2d-2} \cdot \sum_\varphi \lambda_{d-1}(\varphi) \cdot \int\displaylimits_0^{d/\lambda_{d-1}(\varphi)}  \left(1-\frac{\lambda_{d-1}(\varphi) \cdot h}{d}\right)^{n-k} \; {\rm d}h.\]
By substituting $t= \frac{\lambda_{d-1}(\varphi) \cdot h}{d}$, we can rewrite this expression as
\[d^{2d-2} \cdot \sum_\varphi \frac{d \cdot \lambda_{d-1}(\varphi)}{\lambda_{d-1}(\varphi)} \cdot \int_0^1  (1-t)^{n-k} \; {\rm d}t = d^{2d-1} \cdot \sum_\varphi \int_0^1 1 \cdot (1-t)^{n-k} \; {\rm d}t.\]
By Lemma~\ref{lem-integral}, this equals
\[d^{2d-1} \cdot \sum_\varphi \frac{0! \cdot (n-k)!}{(n-k+1)!} = \frac{d^{2d-1}}{n-k+1} \sum_\varphi 1.\]
Since $\conv(\{p_1,\dots,\allowbreak p_{i-1}\})$ is a convex polytope in $\mathbb{R}^d$ with at most $i-1 \leq k$ vertices, 
Theorem~\ref{thm-upperBoundThm} implies that the number of facets $\varphi$ of $\conv(\{p_1,\dots,\allowbreak p_{i-1}\})$ is at most $2 \binom{k}{\lfloor d/2 \rfloor}$.
Altogether, we have derived the desired bound
\[\Pr[E_{a,i} \mid E_{a,i-1}] \leq \frac{2d^{2d-1} \cdot  \binom{k}{\lfloor d/2 \rfloor}}{n-k+1}\]
in the case when $\Delta$ is a regular simplex.

If $\Delta$ is not regular, we first apply a volume-preserving affine transformation $F$ that maps $\Delta$ to a regular simplex $F(\Delta)$.
The simplex $F(\Delta)$ is then contained in the convex body $F(K)$ of volume $1$.
Since $F$ translates the uniform distribution on $F(K)$ to the uniform distribution on $K$ and preserves holes and islands, we obtain the required upper bound also in the general case.
\end{proof}

Now, we finish the proof of Theorem~\ref{thm:islands_2d}.

\begin{proof}[Proof of Theorem~\ref{thm:islands_2d}]
We estimate the expected value of the number $X$ of $k$-islands in $S$.
The number of ordered $k$-tuples of points from $S$ is $n(n-1)\cdots(n-k-1)$.
Since every subset of $S$ of size $k$ admits a unique labeling that satisfies the conditions~\ref{item-canonical0}, ~\ref{item-canonical1}, \ref{item-canonical4}, \ref{item-canonical2}, and \ref{item-canonical3}, we have
\begin{align*}
\mathbb{E}[X] &= n(n-1)\cdots(n-k+1) \cdot \Pr\left[\cup_{a=0}^{k-d-1} E_{a,k}\right] \\
&= n(n-1)\cdots(n-k+1) \cdot \sum_{a=0}^{k-d-1} \Pr\left[E_{a,k}\right],
\end{align*}
as the events $E_{0,k},\dots E_{k-d-1,k}$ are pairwise disjoint.

The probability of the event $L_2$, which says that the points $p_1,\dots,p_d$ satisfy the condition~\ref{item-canonical1}, is $1/d!$.
Let $P = \sum_{a=0}^{k-d-1} \Pr\left[E_{a,k} \mid L_2 \right]$.
For any two events $E,E'$ with $E \supseteq E'$ and $\Pr[E]>0$, we have $\Pr[E'] = \Pr[E \cap E'] = \Pr[E' \mid E]\cdot \Pr[E]$.
Thus, using $L_2 \supseteq E_a = E_{a,d+a+1} \supseteq E_{a,d+a+2} \supseteq \cdots \supseteq E_{a,k}$, we get
\[ \mathbb{E}[X] = n(n-1)\cdots(n-k+1) \cdot \Pr[L_2] \cdot P = \frac{n(n-1)\cdots(n-k+1)}{d!} \cdot P\]
and
\[
P = \sum_{a=0}^{k-d-1}\Pr[E_a \mid L_2]\cdot \prod_{i=d+a+2}^k \Pr[E_{a,i} \mid E_{a,i-1}].
\]

For every $a \in \{d+2,\dots,k-d-1\}$, Lemma~\ref{lem-probabilityInside} gives
\[\Pr[E_a \mid L_2] \leq \frac{2^{d-1}\cdot d!}{(k-a-d-1)!\cdot(n-k+1)^{a+1}} \leq \frac{2^{d-1}\cdot d!}{(n-k+1)^{a+1}}\]
and, due to Lemma~\ref{lem-probabilityOutside},
\[\Pr[E_{a,i} \mid E_{a,i-1}] \leq \frac{2d^{2d-1} \cdot  \binom{k}{\lfloor d/2 \rfloor}}{n-k+1}\]
for every $i \in \{d+a+2,\dots,k\}$.

Using these estimates we derive
\begin{align*}
P & \leq 2^{d-1}\cdot d! \cdot \left(2d^{2d-1}\binom{k}{\lfloor d/2 \rfloor}\right)^{k-d-1} \cdot \sum_{a=0}^{k-d-1}\frac{1}{(n-k+1)^{a+1}} \cdot \frac{1}{(n-k+1)^{k-d-a-1}}\\
&= 2^{d-1}\cdot d! \cdot \left(2d^{2d-1}\binom{k}{\lfloor d/2 \rfloor}\right)^{k-d-1} \cdot (k-d) \cdot \frac{1}{(n-k+1)^{k-d}}.
\end{align*}

Thus the expected number of $k$-islands in $S$ satisfies
\begin{align*}
\mathbb{E}[X] &= \frac{n(n-1) \cdots (n-k+1)}{d!} \cdot P\\
&\leq \frac{2^{d-1}\cdot d! \cdot \left(2d^{2d-1}\binom{k}{\lfloor d/2 \rfloor}\right)^{k-d-1} \cdot (k-d)}{d!} \cdot \frac{n(n-1) \cdots (n-k+1)}{(n-k+1)^{k-d}} \\
&=2^{d-1}\cdot \left(2d^{2d-1}\binom{k}{\lfloor d/2 \rfloor}\right)^{k-d-1} \cdot (k-d) \cdot \frac{n(n-1) \cdots (n-k+2)}{(n-k+1)^{k-d-1}}.
\end{align*}
This finishes the proof of Theorem~\ref{thm:islands_2d}.
\end{proof}

In the rest of the section, we sketch the proof of Theorem~\ref{thm:holes_2d} by showing that a slight modification of the above proof yields an improved bound on the expected number $EH^K_{d,k}(n)$ of $k$-holes in $S$.

\begin{proof}[Sketch of the proof of Theorem~\ref{thm:holes_2d}]
If $k$ points from $S$ determine a $k$-hole in $S$, then, in particular, the simplex $\Delta$ contains no points of $S$ in its interior.
Therefore 
\[EH^K_{d,k}(n) \leq n(n-1)\cdots(n-k+1) \cdot \Pr[E_{0,k}].\]
Then we proceed exactly as in the proof of Theorem~\ref{thm:islands_2d}, but we only consider the case $a=0$.
This gives the same bounds as before with the term $(k-d)$ missing and with an additional factor $\frac{1}{(k-d-1)!}$ from Lemma~9, which proves Theorem~\ref{thm:holes_2d}.
\end{proof}

For $d=2$ and $k=4$, Theorem~\ref{thm:holes_2d} gives $EH^K_{2,4}(n) \leq 128n^2 + o(n^2)$.
We can obtain an even better estimate $EH^K_{2,4}(n) \leq 12n^2 + o(n^2)$ in this case.
First, we have only three facets $\varphi$, as they correspond to the sides of the triangle $\Delta$.
Thus the term $\left(2\binom{k}{\lfloor d/2 \rfloor}\right)^{k-d-1} = 8$ is replaced by $3$.
Moreover, the inequality~\eqref{eq-simplex} can be replaced by 
\[\lambda_1(C_\varphi \cap H(h) \cap \Delta^*) \leq \lambda_1(\varphi),\]
since every line $H(h)$ intersects $R_j \subseteq \Delta^*$ in a line segment of length at most $\lambda_1(F_j) = \lambda(\varphi)$.
This, then, removes the factor $d^{(2d-2)(k-d-1)} = 4$.

\section{Proof of Theorem~\ref{thm:exponential}}
\label{sec:exponentially_k_islands}

For an integer $d \geq 2$ and a convex body $K$ in $\RR^d$ with $\lambda_d(K)=1$.
We show that there are positive constants $C_1=C_1(d)$, $C_2=C_2(d)$, and $n_0=n_0(d)$ such that for every integer $n \geq n_0$ the expected number of islands in a set $S$ of $n$ points chosen uniformly and independently at random from $K$ is at least $2^{C_1 \cdot n^{(d-1)/(d+1)}}$ and at most $2^{C_2 \cdot n^{(d-1)/(d+1)}}$.

For every point set $Q$, there is a one-to-one correspondence between the set of all subsets in $Q$ in convex position and the set of all islands in $Q$.
It suffices to map a subset $G$ of $Q$ in convex position to the island $Q \cap \conv(G)$ in~$Q$.
On the other hand, every island $I$ in $Q$ determines a subset of $Q$ in convex position that is formed by the vertices of $\conv(I)$.
Therefore the number of subsets of $Q$ in convex position equals the number of islands in~$Q$.

For $m \in \NN$, let $p(m,K)$ be the probability that $m$ points chosen uniformly and independently at random from $K$ are in convex position.
We use the following result by B\'{a}r\'{a}ny~\cite{barany01} to estimate the expected number of islands in $S$.

\begin{theorem}[\cite{barany01}]
\label{thm-barany}
For every integer $d \geq 2$, let $K$ be a convex body in $\RR^d$ with $\lambda_d(K)=1$.
Then there are positive constants $m_0$, $c_1$, and $c_2$ depending only on $d$ such that, for every $m \geq m_0$, we have
\[c_1 < m^{2/(d-1)} \cdot (p(m,K))^{1/m} < c_2.\]
\end{theorem}

We let $X$ be the random variable that denotes the number of subsets of $S$ in convex position.
Then
\[\mathbb{E}[X] = \sum_{m=1}^n\binom{n}{m} \cdot p(m,K).\]

Now, we prove the upper bound on the expected number of islands in $S$.
By Theorem~\ref{thm-barany}, we have
\[\mathbb{E}[X] \leq \sum_{m=1}^{m_0-1}\binom{n}{m} + \sum_{m=m_0}^n\binom{n}{m} \cdot \left(\frac{c_2}{m^{2/(d-1)}}\right)^m\]
for some constants $m_0$ and $c_2$ depending only on $d$.
Since $\binom{n}{m} \leq n^m$, the first term on the right side is at most $n^c$ for some constant $c=c(d)$.
Using the bound $\binom{n}{m} \leq (en/m)^m$, we bound the second term from above by 
\[\sum_{m=m_0}^n\ \left(\frac{c_2 \cdot e\cdot n}{m^{(d+1)/(d-1)}}\right)^m.\]
The real function $f(x)=(e \cdot c_2 \cdot n/x^{(d+1)/(d-1)})^x$ is at most $1$ for $x \geq (e\cdot c_2 \cdot n)^{(d-1)/(d+1)}$.
Otherwise $x=y (e\cdot c_2 \cdot n)^{(d-1)/(d+1)}$ for some $y \in [0,1]$ and
\[f(x) = \left(y^{-(d+1)/(d-1)}\right)^{y(e\cdot c_2 \cdot n)^{(d-1)/(d+1)}} = e^{\frac{d+1}{d-1}y\ln{(1/y)}(e\cdot c_2 \cdot n)^{(d-1)/(d+1)}} \leq 2^{c'n^{(d-1)/(d+1)}},\]
where $c'=c'(d)$ is a sufficiently large constant.
Thus $\mathbb{E}[X] \leq n^c + n2^{c'n^{(d-1)/(d+1)}}$.
Choosing $C_2=C_2(d)$ sufficiently large, we have $\mathbb{E}[X] \leq 2^{C_2n^{(d-1)/(d+1)}}$.
Since the number of subsets of $S$ in convex position equals the number of islands in $S$, we have the same upper bound on the expected number of islands in $S$.

It remains to show the lower bound on the expected number of islands in $S$.
By Theorem~\ref{thm-barany}, we have
\[\mathbb{E}[X] \geq \sum_{m=m_0}^n\binom{n}{m} \cdot \left(\frac{c_1}{m^{2/(d-1)}}\right)^m\]
for some constants $m_0$ and $c_1$ depending only on $d$.
Using the bound $\binom{n}{m} \geq (n/m)^m$, we obtain
\[\mathbb{E}[X] \geq \sum_{m=m_0}^n\cdot \left(\frac{c_1 \cdot n}{m^{(d+1)/(d-1)}}\right)^m.\]
Let $C_1=(c_1/2)^{(d-1)/(d+1)}$.
Then, for each $m$ satisfying $C_1 n^{\frac{d-1}{d+1}}/2 \leq m \leq C_1 n^{\frac{d-1}{d+1}}$, the summand in the above expression is at least $2^{C_1 n^{\frac{d-1}{d+1}}/2}$.
It follows that the expected number of $m$-tuples from $S$ in convex position, where $C_1 n^{\frac{d-1}{d+1}}/2 \leq m \leq C_1 n^{\frac{d-1}{d+1}}$, is at least 
\[2^{C_1 n^{(d-1)/(d+1)}/2},\]
as if $n$ is large enough with respect to $d$, then there is at least one $m \in \mathbb{N}$ in the given range.
As we know, each such $m$-tuple in $S$ corresponds to an island in $S$.
Thus the expected number of islands in $S$ is also at least $2^{C_1 n^{(d-1)/(d+1)}/2}$.

\section{Proof of Theorem~\ref{prop-Horton}}
\label{sec:prop-Horton}

Here, for every $d$, we state the definition of a $d$-dimensional analogue of Horton sets on $n$ points from~\cite{VALTR1992b} and show that, for all fixed integers $d$ and $k$, every $d$-dimensional Horton set $H$ with $n$ points contains at least $\Omega(n^{\min\{2^{d-1},k\}})$ $k$-islands in $H$.
If $k \leq 3 \cdot 2^{d-1}$, then we show that $H$ contains at least $\Omega(n^k)$ $k$-holes in $H$.

First, we need to introduce some notation.
A set $Q$ of points in $\RR^d$ is in \emph{strongly general position} if $Q$ is in general position and, for every $i=1,\dots,d-1$, no $(i+1)$-tuple of points from $Q$ determines an $i$-dimensional affine subspace of $\RR^d$ that is parallel to the $(d-i)$-dimensional linear subspace of $\RR^d$ that contains the last $d-i$ axes.
Let $\pi \colon \RR^d \to \RR^{d-1}$ be the projection defined by $\pi(x_1,\dots,x_d) = (x_1,\dots,x_{d-1})$.
For $Q \subseteq \RR^d$, we use $\pi(Q)$ to denote the set $\{\pi(q) \in \RR^{d-1} \colon q \in Q\}$.
If $Q$ is a set of $n$ points $q_0,\dots,q_{n-1}$ from $\RR^d$ in strongly general position that are ordered so that their first coordinates increase, then, for all $m \in \mathbb{N}$ and $i \in \{0,1,\dots,m-1\}$, we define $Q_{i,m}=\{q_j \in Q \colon j \equiv i \; (\bmod \; m)\}$.

For two sets $A$ and $B$ of points from $\RR^d$ with $|A|,|B| \geq d$, we say that $B$ is \emph{deep below} $A$ and $A$ is \emph{high above} $B$ if $B$ lies entirely below any hyperplane determined by $d$ points of $A$ and $A$ lies entirely above any hyperplane determined by $d$ points of $A$.
For point sets $A'$ and $B'$ in $\RR^d$ of arbitrarily size, we say that $B'$ is \emph{deep below} $A'$ and $A'$ is \emph{high above} $B'$ if there are sets $A \supseteq A'$ and $B \supseteq B'$ such that $|A|,|B| \geq d$, $B$ is deep below $A$, and $A$ is high above~$B$.

Let $p_2<p_3<p_4<\cdots$ be the sequence of all prime numbers.
That is, $p_2=2$, $p_3=3$, $p_4=5$ and so on.

We can now state the definition of the $d$-dimensional Horton sets from~\cite{VALTR1992b}.
Every finite set of $n$ points in $\mathbb{R}$ is \emph{$1$-Horton}.
For $d \geq 2$, finite set $H$ of points from $\RR^d$ in strongly general position is a \emph{$d$-Horton set} if it satisfies the following conditions:
\begin{enumerate}[label=(\alph*)]
\item the set $H$ is empty or it consists of a single point, or
\item $H$ satisfies the following three conditions:
\begin{enumerate}[label=(\roman*)]
\item if $d>2$, then $\pi(H)$ is $(d-1)$-Horton,
\item for every $i \in \{0,1,\dots,p_d-1\}$, the set $H_{i,p_d}$ is $d$-Horton,
\item every $I \subseteq \{0,1,\dots,p_d-1\}$ with $|I| \geq 2$ can be partitioned into two nonempty subsets $J$ and $I \setminus J$ such that $\cup_{j \in J}H_{j,p_d}$ lies deep below $\cup_{i \in I \setminus J} H_{i,p_d}$.
\end{enumerate}
\end{enumerate}

Valtr~\cite{VALTR1992b} showed that such sets indeed exist and that they contain no $k$-hole with $k>2^{d-1}(p_2p_3\cdots p_d+1)$.
The $2$-Horton sets are known as \emph{Horton sets}.
We show that $d$-Horton sets with $d \geq 3$ contain many $k$-islands for $k\geq d+1$ and thus cannot provide the upper bound $O(n^d)$ that follows from Theorem~\ref{thm:islands_2d}. 
This contrasts with the situation in the plane, as 2-Horton sets of $n$ points contain only $O(n^2)$ $k$-islands for any fixed~$k$~\cite{FabilaMonroyHuemer2012}.

Let $d$ and $k$ be fixed positive integers.
Assume first that $k \geq 2^{d-1}$. 
We want to prove that there are $\Omega(n^{2^{d-1}})$ $k$-islands in every $d$-Horton set $H$ with $n$ points.
We proceed by induction on $d$. 
For $d=1$ there are $n-k+1=\Omega(n)$ $k$-islands in every $1$-Horton set. 

Assume now that $d>1$ and that the statement holds for $d-1$. 
The $d$-Horton set $H$ consists of $p_d \in O(1)$ subsets $H_{i,p_d}$, each of size at least $\lfloor n/p_d \rfloor \in \Omega(n)$.
The set $\{0,\dots,p_d-1\}$ is ordered by a linear ordering $\prec$ such that, for all $i$ and $j$ with $i \prec j$, the set $H_{i,p_d}$ is deep below $H_{j,p_d}$; see~\cite{VALTR1992b}.
Take two of sets $X = H_{a,p_d}$ and $Y = H_{b,p_d}$ such that $a \prec b$ are consecutive in $\prec$.
Since $k \geq 2^{d-1}$, we have $\lceil k/2 \rceil \geq  \lfloor k/2 \rfloor \geq 2^{d-2}$.
Thus by the inductive hypothesis, the $(d-1)$-Horton set $\pi(X)$ of size at least $\Omega(n)$ contains at least $\Omega(n^{2^{d-2}})$ $\lfloor k/2 \rfloor$-islands.
Similarly, the $(d-1)$-Horton set $\pi(Y)$ of size at least $\Omega(n)$ contains at least $\Omega(n^{2^{d-2}})$ $\lceil k/2 \rceil$-islands. 

Let $\pi(A)$ be any of the $\Omega(n^{2^{d-2}})$ $\lfloor k/2 \rfloor$-islands in $\pi(X)$, where $A\subseteq X$. 
Similarly, let $\pi(B)$ be any of the $\Omega(n^{ 2^{d-2} })$ 
$\lceil k/2 \rceil$-islands in $\pi(Y)$, where $B\subseteq Y$. 
We show that $A \cup B$ is a $k$-island in $H$.
Suppose for contradiction that there is a point $x \in H \setminus (A \cup B)$ that lies in $\conv(A \cup B)$.
Since $a$ and $b$ are consecutive in $\prec$, the point $x$ lies in $X \cup Y = H_{a,p_d} \cup H_{b,p_d}$.
By symmetry, we may assume without loss of generality that $x \in X$.
Since $x \notin A$ and $H$ is in strongly general position, we have $\pi(x) \in \pi(X) \setminus \pi(A)$.
Using the fact that $\pi(A)$ is a $\lfloor k/2 \rfloor$-island in $\pi(X)$, we obtain $\pi(x) \notin \conv(\pi(A))$ and thus $x \notin \conv(A)$.
Since $X$ is deep below $Y$, we have $x \notin \conv(B)$.
Thus, by Carath\'{e}dory's theorem, $x$ lies in the convex hull of a $(d+1)$-tuple $T \subseteq A \cup B$ that contains a point from~$A$ and also a point from~$B$.

We show by induction on $d$ that the affine hull of $U \cap A$ intersects the convex hull of $U \cap B$. 
The sets $(U \cap A) \setminus \{x\}$ and $U \cap B$ partition the vertex set of a $d$-dimensional simplex $S$ that contains $x$ in the interior.
Thus the claim is obviously true for $d=2$.
For $d \geq 3$, the claim is trivial for $|A \cap U|=1$.
If $|U \cap A| \geq 2$, then the affine hull of $A \cap U$ contains a line determined by $x$ and by a vertex $v$ of $S$.
This line intersects the facet $F$ of $S$ opposite to $v$ at some point $y$.
We are done if $F$ does not contain a vertex from $U \cap A$.
Otherwise the induction assumption applied on $F \cup \{y\}$ implies that the affine hull of $((U \cap A) \setminus \{v,x\}) \cup \{y\}$ intersects the convex hull of $U \cap B$, which proves the claim as the affine hull of $((U \cap A) \setminus \{v,x\}) \cup \{y\}$ is contained in the affine hull of $U \cap A$.

Thus the affine hull of $U \cap A$ indeed intersects the convex hull of $U \cap B$. 
Then, however, the set $U \cap A$ is not deep below the set $U \cap B$, which contradicts the fact that $X$ is deep below~$Y$.

Altogether, there are at least $\Omega(n^{ 2^{d-2} })\cdot\Omega(n^{ 2^{d-2} 
})=\Omega(n^{ 2^{d-1} })$ such $k$-islands $A\cup B$, which finishes the proof if $k$ is at least $2^{d-1}$. 
For $k<2^{d-1}$, we use an analogous argument that gives at least $\Omega(n^{\lfloor k/2 \rfloor}) \cdot \Omega(n^{\lceil k/2 \rceil}) = \Omega(n^k)$ $k$-islands in the inductive step.

If $d \geq 2$ and $k \leq 3 \cdot 2^{d-1}$ then a slight modification of the above proof gives $\Omega(n^{\min\{2^{d-1},k\}})$ $k$-islands which are actually $k$-holes in $H$. 
We just use the simple fact that every 2-Horton set with $n$ points contains $\Omega(n^2)$ $k$-holes for every $k \in \{2,\dots,6\}$ as our inductive hypothesis.
This is trivial for $k=2$ and it follows for $k \in \{3,4\}$ from the well-known fact that every set of $n$ points in $\mathbb{R}^2$ in general position contains at least $\Omega(n^2)$ $k$-holes.
For $k \in \{5,6\}$, this fact can be proved using basic properties of 2-Horton sets (we omit the details). 
Then we use the inductive assumption, which says that every $d$-Horton set of $n$ points contains at least $\Omega(n^{\min\{2^{d-1},k\}})$ $k$-holes if $d \geq 2$ and $1 \leq k \leq 3 \cdot 2^{d-1}$.
This finishes the proof of Theorem~\ref{prop-Horton}.

\paragraph{Acknowledgement}

We would like to thank to the referees for their careful reading and comments that helped to improve the presentation of our results.

\bibliography{references}
\bibliographystyle{alphaabbrv-url}

\end{document}